\documentclass[11pt]{article}
\usepackage{fullpage}
\usepackage{amsmath}
\usepackage{amsfonts}
\usepackage{latexsym}
\usepackage{amssymb}
\usepackage{MnSymbol}

\newtheorem{lemma}{Lemma}[section]

\newtheorem{proposition}{Proposition}[section]
\def\homo{{\sf Hom}}
\def\ker{{\sf Ker}}

\def\tor{{\sf tor}}
\def\fin{{\sf fin}}
\def\pt{{\sf PT}}
\def\t{{\sf T}}
\def\min{{\sf MIN}}
\def\ass{{\sf Ass}}

\begin{document}

\newtheorem{definition}{Definition}[section]
\newtheorem{theorem}[definition]{Theorem}
\newtheorem{examples}{Examples}[section]
\newtheorem{corollary}[definition]{Corollary}
\def\square{\Box}
\newtheorem{remark}[definition]{Remark}
\newtheorem{remarks}[definition]{Remarks}
\newtheorem{exercise}[definition]{Exercise}
\newtheorem{example}{Example}[section]
\newtheorem{observation}{Observation}[section]
\newtheorem{observations}{Observations}[section]
\newtheorem{algorithm}[definition]{Algorithm}
\newtheorem{criterion}[definition]{Criterion}
\newtheorem{algcrit}[definition]{Algorithm and criterion}

\newenvironment{prf}[1]{\trivlist
\item[\hskip \labelsep{\it
#1.\hspace*{.3em}}]}{~\hspace{\fill}~$\square$\endtrivlist}
\newenvironment{proof}{\begin{prf}{Proof}}{\end{prf}}

\title{Periodic Behaviors}
\author{
Diego Napp--Avelli\thanks{RD Unit Mathematics and Applications, Department of Mathematics, University of Aveiro, Portugal, \texttt{diego@ua.pt}} \and Marius van der Put\thanks{Institute of Mathematics and Computing Science, University of Groningen, P.O. Box 407, 9700 AK Groningen, The Netherlands, \texttt{mvdput@math.rug.nl}} \and  Shiva Shankar\thanks{Chennai Mathematical Institute, Plot No. H1, SIPCOT, IT Park, Padur P.O., Siruseri, Chennai (Madras)-603103 India, \texttt{sshankar@cmi.ac.in}}
}
\date{}
\maketitle

\begin{center}
\textit{Dedicated to Jan C.~Willems on the occasion of his 70th birthday.}
\end{center}

\begin{abstract} \noindent This paper studies behaviors that are defined on a torus, or equivalently, behaviors defined in spaces of periodic functions, and establishes their basic properties analogous to classical results of Malgrange, Palamodov, Oberst et al. for behaviors on $\mathbb{R}^n$. These properties - in particular the Nullstellensatz describing the {\it Willems closure} - are closely related to integral and rational points on affine algebraic varieties.
\end{abstract}

\section*{Introduction}

In classical control theory the structure of a linear lumped dynamical system, considered as an input-output system, is determined by its frequency response, i.e. its response to periodic inputs. This idea is the foundation of the subject of {\it frequency domain analysis} and the work of Bode, Nyquist and others, and is also the idea underpinning the theory of transfer functions, including its generalization to multidimensional systems \cite{ob3,q,ss,z}.

The more recent Behavioral Theory of J.C.~Willems challenges the notion of an open dynamical system as an input-output system \cite{w}. Instead, a system is considered to be the collection of all signals that can occur and which are therefore the signals that obey the laws of the system. This collection of signals, called the behavior of the system, is the system itself, and is analogous to Poincar\'e's notion of the phase portrait of a vector field. Notions of causality and the related input-output structure are not part of the primary description, but are secondary structures to be imposed only if necessary. The behavioral theory can be seen as a  generalization of the Kalman State Space Theory, and the ideas of state space theory, as well as those of frequency domain can be carried over to the more general situation of behaviors.
It is the purpose of this paper to initiate the study of frequency domain ideas in the theory of distributed behaviors.

A second motivation for this paper is the following. The theory of behaviors has so far been developed for signal spaces that live on 
the `base space' $\mathbb{R}^n$, or on its convex subsets. The commuting global vector fields $\partial _1,\dots ,\partial _n$ generate the algebra 
$\mathbb{C}[\partial_1, \ldots , \partial_n]$ of differential operators with constant coefficients, and distributed behaviors are defined by equations
whose terms are from this algebra. This paper considers the case where the base space $\mathbb{R}^n$ is replaced by a
torus $\mathbb{R}^n/\Lambda$, with $\Lambda$ a lattice. Functions on the torus can be identified with $\Lambda$-invariant
functions on $\mathbb{R}^n$, in other words, functions which are periodic with respect to $\Lambda$.  The torus is an example of a parallelizable manifold; other manifolds of this type, such as 
the 3-sphere $S^3$, would be of interest for behavior theory. Another possibly interesting base space for behavior theory is 
$\mathbb{P}^n(\mathbb{R})$, the real $n$-dimensional projective space. The vector space of global vector fields on this projective space is isomorphic to the Lie algebra $\frak{sl}_{n+1}$ and its enveloping algebra acts as a ring of differential operators on the space of smooth functions on $\mathbb{P}^n(\mathbb{R})$.
 
In this paper we consider the real torus $\t:=\mathbb{R}^n/2\pi \mathbb{Z}^n$. Now $\mathcal{C}^\infty(\t)$, the space of smooth functions on the torus $\t$, is identified with the space of smooth functions on $\mathbb{R}^n$ having the lattice $2\pi \mathbb{Z}^n$ as its group of periods. It is a Fr\'echet space under the topology of uniform convergence of functions and all their derivatives.
On it acts the ring of constant coefficient partial differential operators $\mathcal{D}:=\mathbb{C}[\partial _1,\dots ,\partial _n]$, and makes $\mathcal{C}^\infty (\t)$ a topological $\mathcal{D}$-module. The aim of this paper is to develop the basic properties of system theory in this situation.  It turns out that behaviors, contained in $\mathcal{C}^\infty(\t)^q$, are related to integral points on algebraic varieties in $\mathbb{A}^n$. A comparison with the fundamental paper \cite{ob} is rather useful.
 
Functions which are periodic with respect to the lattice $2\pi \mathbb{Z}^n$
remain periodic with respect to lattices which are integral multiples of this lattice. Thus, one can relax the condition of periodicity with respect to $2\pi \mathbb{Z}^n$ by considering smooth functions on $\mathbb{R}^n$ which are periodic with respect to a lattice $N2\pi \mathbb{Z}^n$ for some integer $N\geq 1$, depending on the function. This space of periodic functions, denoted by  $\mathcal{C}^\infty (\pt)$, can be naturally identified with a dense subspace of the space of continuous functions on the inverse limit $\pt:=\underset{\leftarrow}{\lim}\ \mathbb{R}^n/N2\pi \mathbb{Z}^n$, which we call a {\it protorus}.  Further, $\mathcal{C}^\infty (\pt)$ is 
the strict direct limit of the Fr\'echet spaces $\mathcal{C}^\infty (\mathbb{R}^n/N2\pi \mathbb{Z}^n)$; it is therefore a barrelled and bornological 
topological vector space, and is also a topological $\mathcal{D}$-module.  
 
In the situation of this protorus $\pt$, behaviors are related to rational points of algebraic varieties in $\mathbb{A}^n$. 
We consider various choices of signal spaces, their injectivity (or their injective envelopes) as $\mathcal{D}$-modules and make explicit computations of the associated 
Willems closure for submodules of $\mathcal{D}^q$.  For the 1D case the results are elementary. For the more important nD case (with $n>1$)  the Willems closure is explicitly given for various choices of signal spaces. This 
involves the knowledge of the existence  of (many) rational points or integral points
on algebraic varieties over $\mathbb{Q}$ or $\mathbb{Z}$.  This connection between periodic behaviors
and arithmetic algebraic geometry (diophantine problems) is rather surprising.  

\section{Behaviors and the Willems Closure}
As in the introduction let $\mathcal{D}=\mathbb{C}[\partial _1,\dots ,\partial _n]$. Let $D_j=\frac{1}{\imath}\partial_j, ~j=1, \ldots ,n$, so that also $\mathcal{D}=\mathbb{C}[D_1,\ldots ,D_n]$. 
We consider
a faithful $\mathcal{D}$-module $\mathcal{F}$, i.e. a module having the property that if $r\in \mathcal{D}$ and 
$r\mathcal{F}=0$, then $r=0$. This module is now taken as the signal space. We recall the usual set up for behaviors.

 Let $e_1,\dots ,e_q$ be the standard basis of $\mathcal{D}^q$.
Associate to a submodule $\mathcal{M}\subset \mathcal{D}^q$ its behavior 
$\mathcal{M}^\perp \subset \mathcal{F}^q$ consisting of all elements $(f_1,\dots ,f_q)\in
\mathcal{F}^q$ satisfying $\sum r_j(f_j)=0$ for all $\sum r_je_j\in \mathcal{M}$. In other words,
$\mathcal{M}^\perp$ is the image of the map
${\homo}_{\mathcal{D}}(\mathcal{D}^q/\mathcal{M},\mathcal{F})\rightarrow \mathcal{F}^q$, given by 
$\ell \mapsto (\ell (\bar{e}_1),\dots ,\ell (\bar{e}_q))$, where $\bar{e}_j$ is the class of $e_j$ in $\mathcal{D}^q/\mathcal{M}$. The above defines the set of behaviors $\mathcal{B}\subset \mathcal{F}^q$. For a behavior $\mathcal{B}$, define $\mathcal{B}^\perp :=\{r=\sum r_je_j\in \mathcal{D}^q|\
\sum r_j(f_j)=0\mbox{ for all } (f_1,\dots ,f_q)\in \mathcal{B}\}$.

For any behavior $\mathcal{B}$ it follows that $\mathcal{B}^{\perp \perp }=
\mathcal{B}$. 
The {\it Willems closure} of a submodule $\mathcal{M}\subset \mathcal{D}^q$ with respect to $\mathcal{F}$ is, by definition,  
$\mathcal{M}^{\perp \perp}\subset \mathcal{D}^q$ \cite{ss1}. Clearly $\mathcal{M} \subset \mathcal{M}^{\perp \perp}$. It is well known that $\mathcal{M}^{\perp \perp}=\mathcal{M}$ holds if
 the signal space $\mathcal{F}$ is an injective cogenerator. For more general
 signal spaces one has the following.
 
 \begin{lemma} $\mathcal{M}^{\perp \perp}/\mathcal{M}=\{\xi \in \mathcal{D}^q/\mathcal{M}\ |\ \ell (\xi )=0\mbox{ for all }
 \ell \in {\homo}_\mathcal{D}(\mathcal{D}^q/\mathcal{M},\mathcal{F})\}$. Moreover, $\mathcal{M}^{\perp\perp}/\mathcal{M}$
is a submodule of the torsion module $(\mathcal{D}^q/\mathcal{M})_{\tor}$ of $\mathcal{D}^q/\mathcal{M}$ (where
 $(\mathcal{D}^q/\mathcal{M})_{\tor}:=\{\eta \in \mathcal{D}^q/\mathcal{M}\ |\exists r \in \mathcal{D}, \ r\neq 0, \
 r\eta =0\}$. \end{lemma} 
\begin{proof} By the above definition, $\eta =\sum \eta _je_j \in \mathcal{M}^{\perp \perp}$ if and only if $\sum \eta _j\ell (\bar{e}_j)=0$ for every $\ell$ in ${\homo}_{\mathcal{D}}(\mathcal{D}^q/\mathcal{M},
\mathcal{F})$. The latter is equivalent to $\ell (\sum \eta _j\bar{e}_j)=0$ for all
$\ell \in {\homo}_\mathcal{D}(\mathcal{D}^q/\mathcal{M},\mathcal{F})$.

Define the torsion free module $\mathcal{Q}$ by the exact sequence
\[0\rightarrow (\mathcal{D}^q/\mathcal{M})_{\tor}\rightarrow \mathcal{D}^q/\mathcal{M}\rightarrow \mathcal{Q}\rightarrow 0.\] 
To show that $\mathcal{M}^{\perp \perp }/\mathcal{M}\subset (\mathcal{D}^q/\mathcal{M})_{\tor}$ amounts to showing that
for every non zero element $\xi \in \mathcal{Q}$ there exists a homomorphism
$\ell :\mathcal{Q}\rightarrow \mathcal{F}$ with $\ell (\xi )\neq 0$. As $\mathcal{Q}$ is torsion free it
is a submodule of $\mathcal{D}^r$ for some $r$, and it therefore suffices to verify the above property for $\mathcal{D}$
itself. This amounts to showing that for every $r\in \mathcal{D},\ r\neq 0$, there exists an 
element $f\in \mathcal{F}$ with $r(f) \neq 0$. But this is just the assumption that $\mathcal{F}$ is a faithful $\mathcal{D}$-module. \end{proof}

\begin{corollary} Suppose either that the signal space $\mathcal{F}$ is injective, or
that the exact sequence 
$0\rightarrow (\mathcal{D}^q/\mathcal{M})_{\tor}\rightarrow \mathcal{D}^q/\mathcal{M}\rightarrow \mathcal{Q}\rightarrow 0$
splits. Then $\mathcal{M}^{\perp\perp}/\mathcal{M}$ consists of the elements $\xi \in (\mathcal{D}^q/\mathcal{M})_{\tor}$
such that $\ell (\xi )=0$ for every $\ell \in {\homo}_{\mathcal{D}}((\mathcal{D}^q/\mathcal{M})_{\tor},\mathcal{F})$. 
\end{corollary}
\begin{proof} In both the cases, every homomorphism $\ell :(\mathcal{D}^q/\mathcal{M})_{\tor}\rightarrow 
\mathcal{F}$ extends to an element of ${\homo}_\mathcal{D}(\mathcal{D}^q/\mathcal{M},\mathcal{F})$.
\end{proof}

\begin{corollary} Consider two signal spaces $\mathcal{F}_0\subset \mathcal{F}$.
Assume that for every $a\in \mathcal{F},\ a\neq 0$, there exists a 
homomorphism $m:\mathcal{F}\rightarrow \mathcal{F}_0$ such that $m(a)\neq 0$.
Then the Willems closure of $\mathcal{M}$ with respect to $\mathcal{F}_0$ equals that with respect to $\mathcal{F}$. 
\end{corollary}
\begin{proof} Consider $\xi \in \mathcal{D}^q/\mathcal{M}$. If there exists a homomorphism 
$\ell :\mathcal{D}^q/\mathcal{M}\rightarrow \mathcal{F}$ with $\ell (\xi )\neq 0$, then, by  assumption,
there exists a homomorphism $\tilde{\ell}:\mathcal{D}^q/\mathcal{M}\rightarrow \mathcal{F}_0$ with
$\tilde{\ell}(\xi )\neq 0$. Since the converse of this statement is obvious, the two Willems closures of $\mathcal{M}$ coincide. \end{proof}

See also \cite{wd} for related results.

\section{Periodic Functions and the Protorus}
We consider, as in the introduction, the torus $\t:=\mathbb{R}^n/2\pi \mathbb{Z}^n$. An element $f:\t \rightarrow \mathbb{C}$ 
of $\mathcal{C}^\infty (\t)$ is represented by its Fourier series: $f(x) = 
\sum _{a\in \mathbb{Z}^n} c_a
e^{\imath<a,x>}$, where 
$a=(a_1,\dots ,a_n)$, $x=(x_1,\dots ,x_n)$ and
$<a,x>=\sum a_jx_j$. Further, the coefficients
$c_a\in \mathbb{C}$ are required to satisfy the property:
for every integer $k\geq 1$ there exists a constant $C_k>0$ such that
$|c_a|\leq \frac{C_k}{(1+\sum _{j=1}^n|a_j|)^k}$ for all $a$.
(We note that the space of distributions on $\t$ has a similar description, however with different requirements on the absolute values $|c_a|$.)

The vector space $\mathcal{C}^\infty (\t)=\mathcal{C}^\infty (\mathbb{R}^n/2\pi \mathbb{Z}^n)$ has the natural structure of a Fr\'echet space, moreover it is a topological $\mathcal{D}$-module. For positive integers $N_1$ dividing $N_2$,
the natural $\mathcal{D}$-module morphism  $\mathcal{C}^\infty (\mathbb{R}^n/2\pi N_1\mathbb{Z}^n)\rightarrow \mathcal{C}^\infty (\mathbb{R}^n/ 2\pi N_2\mathbb{Z}^n)$ identifies the first linear topological 
space with a closed subspace of the second one. We define 
$\mathcal{C}^\infty (\pt):=\underset{\rightarrow }{\lim} \ \mathcal{C}^\infty (\mathbb{R}^n/2\pi N\mathbb{Z}^n)$. This is a strict direct limit of 
Fr\'echet spaces, and is a locally convex bornological and barrelled topological vector space. The elements of $\mathcal{D}$ act continuously on it so that $\mathcal{C}^\infty(\pt)$ is also a topological $\mathcal{D}$-module. An element $f$ in it is
represented by the series $f(x)=\sum _{a\in \mathbb{Q}^n} c_a e^{\imath<a,x>}$, where the {\it support} of $f$, i.e.,
$\{a\in \mathbb{Q}^n|\ c_a\neq 0\}$, is a subset of
$\frac{1}{N}\mathbb{Z}^n$ for some integer $N\geq 1$, depending on $f$. Further, there is the same requirement of rapid decrease on the absolute values $|c_a|$ as above. 

As in the Introduction, call the inverse limit $\pt := \underset{\leftarrow}{\lim }\ \mathbb{R}^n/2\pi N\mathbb{Z}^n$ a {\it protorus}.
$\pt$ is a compact topological group. The map $\pt \rightarrow \mathbb{R}^n/2\pi N\mathbb{Z}^n$ embeds $\mathcal{C}^\infty (\mathbb{R}^n/2\pi N\mathbb{Z}^n)$ in the space $\mathcal{C}(\pt)$ of continuous functions on the protorus (which is a Banach space with respect to the sup norm) for every $N$. The exact sequence 
\[
0 \rightarrow 2\pi \mathbb{Z}^n/2\pi N\mathbb{Z}^n \rightarrow \mathbb{R}^n/2\pi N\mathbb{Z}^n \rightarrow \mathbb{R}^n/2\pi \mathbb{Z}^n \rightarrow 0
\]
for each $N$, gives upon taking inverse limits the exact sequence
\[
0\rightarrow \widehat{Z}^n\rightarrow \pt \rightarrow \mathbb{R}^n/2\pi \mathbb{Z}^n \rightarrow 0
\]
where the group $\underset{\leftarrow}{\lim}\ 2\pi \mathbb{Z}^n/2\pi N\mathbb{Z}^n$ equals $\widehat{Z}^n$, $\widehat{Z}$ being the well known profinite completion $\underset{\leftarrow}{\lim}\ \mathbb{Z}/N\mathbb{Z}$. 
$\widehat{Z}^n$ sits inside the protorus $\pt$ as a compact subgroup and is totally disconnected. This implies
that any continuous map $\widehat{Z}^n\rightarrow \mathcal{C}(\pt)$ is the uniform limit of locally constant maps.
 
For $f\in \mathcal{C}(\pt)$ and $z\in \widehat{Z}^n$,
define the function $f_z$ by $f_z(t)=f(z+t)$. The map $z\mapsto f_z$ is 
continuous and therefore a uniform limit of locally constant maps.
Thus $f$ is the uniform limit of functions $f_i$ in $\mathcal{C}(\pt)$, where 
$z \mapsto (f_i)_z$ is locally constant. This implies that 
each $f_i$ is invariant under the shift $N\widehat{Z}^n$ for some integer $N\geq 1$, depending on $i$; in other words $f_i$ is an element of 
$\mathcal{C}(\mathbb{R}^n/2\pi N\mathbb{Z}^n)$, the space of continuous complex valued functions on $\mathbb{R}^n/2\pi N\mathbb{Z}^n$.
As 
$\mathcal{C}^\infty(\mathbb{R}^n/2\pi N\mathbb{Z}^n)$ is dense in 
$\mathcal{C}(\mathbb{R}^n/2\pi N\mathbb{Z}^n)$, it follows that $\mathcal{C}^\infty(\pt)$ is a {\it dense subspace} of $\mathcal{C}(\pt)$.

As the partial sums of a Fourier series expansion converge uniformly, it follows that for $L(D)$ in $\mathcal{D}$, 

\[ L(D)(\sum _{a\in \mathbb{Q}^n} c_a
e^{\imath<a,x>})=
\sum _{a \in \mathbb{Q}^n} c_a L(a)
e^{\imath<a,x>}\]
The basic observation, leading to the computation of the Willems closure is
that $L(D)$ is injective on $\mathcal{C}^\infty (\pt)$ if and only if the polynomial equation
$L(a)=L(a_1,\dots ,a_n)=0$ has no solutions in $\mathbb{Q}^n$.
(We note, in passing, that the condition $L(a_1,\dots ,a_n)\neq 0$ for
$(a_1,\dots ,a_n)\in \mathbb{Q}^n$ does not imply that $L(D)$ is surjective; see 
Theorem 2.1.)

Another observation is that $\mathcal{C}^\infty(\pt)$ is not an injective $\mathcal{D}$-module, not even a divisible
module. Indeed,
the image of the morphism $D_1: \mathcal{C}^\infty (\pt)\rightarrow \mathcal{C}^\infty (\pt)$ consists of those 
elements $f$ whose support is contained in $\{(a_1,\dots ,a_n)\in \mathbb{Q}^n|\ 
a_1\neq 0\}$. The kernel of $D_1$ is  the subspace of
$\mathcal{C}^\infty (\pt)$ consisting of those elements $f$ whose support lies in
$\{(a_1,\dots ,a_n)\in \mathbb{Q}^n|\ a_1=0\}$. The cokernel of the morphism $D_1$ is represented by this same subspace of $C^\infty (\pt)$, it is therefore not surjective. 

We also consider the subalgebra $\mathcal{C}^\infty(\pt)[x_1,\dots ,x_n]$ of
$\mathcal{C}^\infty(\mathbb{R}^n)$ obtained by adjoining the elements $x_1,
\dots ,x_n$, that is the coordinate functions, to $\mathcal{C}^\infty(\pt)$, 
and similarly $\mathcal{C}^\infty(\t)[x_1,\dots ,x_n]$ etc. Yet another 
observation is

\begin{lemma} $\mathcal{C}^\infty(\pt)[x_1,\dots ,x_n] = \bigoplus_{a\in \mathbb{N}^n} \mathcal{C}^\infty(\pt)x_1^{a_1}\dots x_n^{a_n}$, where $a=(a_1,\dots ,a_n)$; and similarly $\mathcal{C}^\infty(\t)[x_1,\dots ,x_n] = \bigoplus_{a\in \mathbb{N}^n} \mathcal{C}^\infty(\t)x_1^{a_1}\dots x_n^{a_n}. $   
\end{lemma}  
\begin{proof} Clearly $\mathcal{C}^\infty(\pt)[x_1,\dots ,x_n] = \sum_{a\in \mathbb{N}^n} \mathcal{C}^\infty(\pt)x_1^{a_1}\dots x_n^{a_n}$, so it remains
to show that the sum is direct.

We first observe that $\mathcal{C}^\infty(\pt)[x_1] = \bigoplus_{a\in \mathbb{N}}\mathcal{C}^\infty(\pt)x_1^a$, for if not, there would be a relation $\sum_{a\in \mathbb{N}}f_ax_1^a = 0$, with finitely many of the $f_a$ nonzero. Suppose $f_0$ is nonzero; then the above relation implies that $f_0 = -\sum_{a>0}f_ax_1^a$. This is a contradiction because $f_0$ is in $\mathcal{C}^\infty(\pt)$ while the sum on the right hand side is not. Thus $f_0 = 0$. This implies that the relation above is of the form 
$x_1(\sum_{a>0}f_ax_1^{a-1})=0$. As the function $x_1$ is zero only on a set of measure 0, it follows that $\sum_{a>0}f_ax_1^{a-1}=0$, leading to a contradiction just as above.

Suppose now by induction that $\mathcal{C}^\infty(\pt)[x_1,\dots ,x_{n-1}] =
\bigoplus_{a\in \mathbb{N}^{n-1}}\mathcal{C}^\infty(\pt)x_1^{a_1}\dots x_{n-1}^{a_{n-1}}$, 
and suppose that $\mathcal{C}^\infty(\pt)[x_1,\dots ,x_{n-1},x_n]=\mathcal{C}^\infty(\pt)[x_1,\dots ,x_{n-1}][x_n]$ 
is not a direct sum. Then there is a relation $\sum_{a\in \mathbb{N}}f_ax_n^a = 0$, 
with finitely many of the $f_a$ (in $\mathcal{C}^\infty(\pt)[x_1,\dots ,x_{n-1}]$) nonzero. This again leads to a contradiction as above. Thus $\mathcal{C}^\infty[x_1,\dots ,x_n] = \bigoplus_{a \in \mathbb{N}}\mathcal{C}^\infty(\pt)[x_1,\dots ,x_{n-1}]x_n^a = \bigoplus_{a\in \mathbb{N}^n}\mathcal{C}^\infty(\pt)x_1^{a_1}\dots x_n^{a_n}$. 
\end{proof}  
This lemma allows us to write an element in $\mathcal{C}^\infty(\pt)[x_1,\dots ,x_n]$ {\it uniquely} as a polynomial in the $x_i$'s with coefficients in $\mathcal{C}^\infty(\pt)$.

Define $\mathcal{C}^\infty (\pt)_{\fin}$ to be the $\mathcal{D}$-submodule of $\mathcal{C}^\infty(\pt)$ consisting of those elements $f$ with finite support, i.e. those elements whose Fourier series expansion is a finite sum. Just as above, $\mathcal{C}^\infty (\pt)_{\fin}$
is not an injective $\mathcal{D}$-module. However, the following proposition gives an explicit expression for its injective envelope.

\begin{proposition} The $\mathcal{D}$-module $\mathcal{C}^\infty (\pt)_{\fin}[x_1,\dots ,x_n]$ is an injective envelope of $\mathcal{C}^\infty (\pt)_{\fin}$. Similarly, $\mathcal{C}^\infty (\t)_{\fin}[x_1,\dots ,x_n]$ is 
an injective envelope of $\mathcal{C}^\infty(\t)_{\fin}$. 
 \end{proposition}
\begin{proof} The Fundamental Principle of Malgrange - Palamodov states that
$\mathcal{C}^\infty (\mathbb{R}^n)$ is an injective $\mathcal{D}$-module. It is also a cogenerator (Oberst \cite{ob}). From this it follows that its submodule 
$\min:=\mathbb{C}[\{e^{\imath<a,x>}\}_{a\in \mathbb{C}^n},
x_1,\dots ,x_n]$ is the direct sum of the injective envelopes $E(\mathcal{D}/\frak{m})$ of the modules $\mathcal{D}/\frak{m}$,
where $\frak{m}$ varies over the set $\{ (D_1-a_1, \ldots , D_n-a_n), a=(a_1, \ldots ,a_n) \in \mathbb{C}^n \}$ of maximal ideals of $\mathcal{D}$. Thus this
module is again injective, and is in fact a minimal injective cogenerator over
$\mathcal{D}$, unique up to isomorphism (see \cite{ob2} for more details). The elements of $\min$ are  finite sums 
$\sum _{a\in \mathbb{C}^n}p_a(x)
e^{\imath<a,x>}$, where the $p_a(x)$
are polynomials in $x_1,\dots ,x_n$. Define the map 
$\pi: \min \rightarrow \mathcal{C}^\infty (\pt)_{\fin}[x_1,\dots ,x_n]$ by 
\[\pi (\sum _{a\in \mathbb{C}^n}p_a(x)
e^{\imath<a,x>}) =\sum _{a\in \mathbb{Q}^n}p_a(x) e^{\imath<a,x>}\]

Clearly $\pi$ is a $\mathbb{C}$-linear projection; it also
commutes with the operators $D_j, \ j=1,\dots ,n$. Thus $\pi$
is a morphism of $\mathcal{D}$-modules which splits the inclusion 
$i:\mathcal{C}^\infty(\pt)_\fin[x_1,\dots,x_n] \rightarrow \min$. It follows that $\mathcal{C}^\infty (\pt)_{\fin}[x_1,\dots ,x_n]$ is a direct summand of $\min$, hence an injective $\mathcal{D}$-module. Moreover, the extension of
modules $\mathcal{C}^\infty (\pt)_{\fin}\subset \mathcal{C}^\infty (\pt)_{\fin}[x_1,\dots ,x_n]$ is essential. Indeed, consider a term $f=x_1^{m_1}\cdots x_n^{m_n}
e^{\imath<a,x>}$ with $a\in \mathbb{Q}^n$. As $(D_j-a_j)(x_je^{\imath <a,x>})=\frac{1}{\imath}e^{\imath <a,x>}$, it follows that
$(D_1-a_1)^{m_1}\cdots (D_n-a_n)^{m_n}(f)=ce^{\imath<a,x>}$, for some nonzero constant $c$. Thus we conclude that
 $\mathcal{C}^\infty (\pt)_{\fin}[x_1,\dots ,x_n]$ is an injective envelope of $\mathcal{C}^\infty (\pt)_{\fin}$. \end{proof}

\begin{observations} {\rm (1) $\mathcal{C}^\infty (\pt)\subset \mathcal{C}^\infty (\pt)[x_1,\dots ,x_n]$ is {\it not} an essential extension. Indeed, consider
 $f=x_1\sum _{a\in \mathbb{Z}^n} c_ae^{\imath<a,x>}$ in $\mathcal{C}^\infty (\pt)[x_1,\dots ,x_n]$ 
 with $c_a\in \mathbb{C}$ and all
 $c_a\neq 0$. For any $L(D)\in \mathcal{D}$,  
 $L(D)f=x_1\sum c_aL(a_1,\dots ,a_n) e^{\imath<a,x>}+$
 (an element of $\mathcal{C}^\infty (\pt)$). Thus $L(D)f\in \mathcal{C}^\infty (\pt)$ implies $L=0$ (no nonzero polynomial can vanish at every integral point).\\
(2) The polynomials in $x_1,\dots ,x_n$ have no interpretation as functions on the
protorus $\pt$, but are functions on the space $\mathbb{R}^n$, which  can be seen as the universal covering  of the protorus.} \end{observations}
  
\begin{lemma} Let $n=1$. Then $\mathcal{C}^\infty (\t)[x]$ is an injective $\mathcal{D}=\mathbb{C}[D]$-module, where $D=\frac{1}{\imath}\frac{d}{dx}$. 
Thus for $a \not \in \mathbb{Z}$, the map $D-a$ is bijective on
$\mathcal{C}^\infty (\t)[x]$. For $a \in \mathbb{Z}$ the kernel of 
$D-a$ on $\mathcal{C}^\infty (\t)[x]$ is $\mathbb{C}e^{\imath ax}$.

There is exactly one injective envelope of $\mathcal{C}^\infty (\t)$ in $\mathcal{C}^\infty (\t)[x]$, and it consists of the elements $\sum _{j\geq 0}f_jx^j$ such that $f_j$ has finite support for $j\geq 1$. 

Similar statements hold for $\mathcal{C}^\infty(\pt)$ replacing $\mathcal{C}^\infty (\t)$ and $\mathbb{Q}$ replacing $\mathbb{Z}$.
\end{lemma}
\begin{proof} Since $n=1$, injectivity is equivalent to divisibility. Thus it suffices to
show that $(D-a): \mathcal{C}^\infty (\t)[x]\rightarrow \mathcal{C}^\infty (\t)[x]$ is surjective for every $a \in \mathbb{C}$. But if 
$g=\sum_{j=1}^kg_jx^j$ is an element of $\mathcal{C}^\infty(\t)[x]$, where 
the $g_j$ are in $\mathcal{C}^\infty(\t)$, then an 
$f$ such that $(D-a)f=g$ is, by the `variation of constants' formula, 
given by $f(x)=e^{\imath ax}\int_0^xe^{-\imath at}g(t)dt$, which is again in $\mathcal{C}^\infty(\t)[x]$.  

Now, the theory of Matlis \cite{mt} applied to the case of this injective module $\mathcal{C}^\infty (\t)[x]$, states that it admits a decomposition
\[
\mathcal{C}^\infty (\t)[x] = \bigoplus_{a\in \mathbb{Z}}\mathbb{C}[x]e^{\imath ax} \bigoplus \mathcal{V}
\]
where the torsion module $\tor (\mathcal{C}^\infty (\t)[x])$ of $\mathcal{C}^\infty (\t)[x]$ equals $\bigoplus_{a\in \mathbb{Z}}\mathbb{C}[x]e^{\imath ax}$ and where the module $\mathcal{V}\simeq \mathcal{C}^\infty (\t)[x]/\tor (\mathcal{C}^\infty (\t)[x])$ is injective and torsion free (see also \cite{ob2}). In general $\mathcal{V}$ is not unique, and one can only speak of {\it an} injective envelope of $\mathcal{C}^\infty(\t)$ in $\mathcal{C}^\infty (\t)[x]$; nonetheless it turns out for the case at hand that there is exactly one injective envelope as described in the statement.

This follows from the fact that an injective envelope of $\mathbb{C}e^{\imath ax}$ is $\mathbb{C}[x]e^{\imath ax}$; thus as $\mathbb{C}e^{\imath ax}$ is contained in $\mathcal{C}^\infty(\t)$, the above decomposition implies that any injective envelope $\mathcal{E}$ of $\mathcal{C}^\infty(\t)$ in $\mathcal{C}^\infty (\t)[x]$ must satisfy
\[
\bigoplus_{a\in \mathbb{Z}}\mathbb{C}[x]e^{\imath ax} + \mathcal{C}^\infty (\t) \subseteq \mathcal{E}=\bigoplus_{a\in \mathbb{Z}}\mathbb{C}[x]e^{\imath ax} \bigoplus (\mathcal{V}\bigcap\mathcal{E})
\]
But if an element $f=\sum _{j=0}^kf_jx^j$ in
$\mathcal{C}^\infty (\t)[x]$ belongs to $\mathcal{E}$, then $0 \neq L(D)f \in \mathcal{C}^\infty (\t)$ for some $L(\mathcal{D})$ in $\mathcal{D}$.  
Now suppose that $k\geq 1$. Since  
$L(D)f=(L(D)f_k)x^k+$ (terms of lower degree in $x$), it follows that $L(D)f_k=0$ and therefore that $f_k$ has finite support  $\{a_1,\dots ,a_s\}$. Then
$M(D):=(D-a_1)\cdots (D-a_s)$ satisfies  $M(D)f_k=0$. After replacing $f$ by 
$M(D)f$, induction with respect to $k$ implies that $f_1, \dots ,f_k$ all have finite support. Thus 
\[
\bigoplus_{a\in \mathbb{Z}}\mathbb{C}[x]e^{\imath ax} + \mathcal{C}^\infty (\t) \subseteq \mathcal{E} \subseteq \bigoplus_{a\in \mathbb{Z}}\mathbb{C}[x]xe^{\imath ax}\bigoplus \mathcal{C}^\infty(\t)  
\]
which implies equality throughout. This proves the second statement.  
  
The corresponding statements for the protorus follow from the fact that $\mathcal{C}^\infty(\pt)$ is the union of its
subspaces $\mathcal{C}^\infty (\mathbb{R}/N2\pi \mathbb{Z}), N \geq 1$.    \end{proof}

\begin{proposition} The spaces $\mathcal{C}^\infty (\t)_{\fin}[x_1,\dots ,x_n] \subset \mathcal{C}^\infty(\t) [x_1,\dots ,x_n]$
define the same Willems closure. The same holds for the inclusion of the two signal spaces
$\mathcal{C}^\infty (\t)_{\fin}\subset \mathcal{C}^\infty(\t)$. These statements remain valid for $\pt$ replacing $\t$.
 \end{proposition}
\begin{proof} For a $b\in \mathbb{Z}^n$, define the homomorphism
\[m_b:\mathcal{C}^\infty(\t) [x_1,\dots ,x_n]\rightarrow  \mathcal{C}^\infty(\t)_{\fin}[x_1,\dots ,x_n]\]
 by $m_b (\sum _{a\in \mathbb{Z}^n}
p_a(x)e^{\imath<a,x>} )= p_b(x)e^{\imath<b,x>}.$ The first statement now follows from Corollary 1.2. The other cases are similar. \end{proof}

\begin{theorem} For $n>1$, the $\mathcal{D}$-modules $\mathcal{C}^\infty(\t)[x_1,\dots ,x_n]$ and 
 $\mathcal{C}^\infty (\pt)[x_1,\dots ,x_n]$ are not divisible (and therefore not injective).
 \end{theorem}
 \begin{proof} It suffices to show that  $\mathcal{C}^\infty(\t)[x_1,x_2]$ is not divisible. Towards this let $\ell$ be any Liouville number, and consider 
$L=D_1+\ell D_2$ in $\mathcal{D}$. Let $g = \sum _{a\in \mathbb{Z}^2} c_a e^{\imath<a,x>}$ be any element in $\mathcal{C}^\infty(\t)$, so that for every integer $k\geq 1$, there is a constant $C_k$ such that $|c_a|\leq C_k(1+|a_1|+|a_2|)^{-k}$ holds for all $a \in \mathbb{Z}^2$. If $\mathcal{C}^\infty(\t)[x_1,x_2]$ were divisible, then $L$ would define a surjective morphism on it, and so there would be an element $f= \sum _{a\in \mathbb{Z}^2} p_a(x)
e^{\imath<a,x>}$ in it such that $L(f)=g$. Thus
\[
\sum_{a\in \mathbb{Z}^2}(D_1p_a(x) + \ell D_2p_a(x) + (a_1+\ell a_2)p_a(x))e^{\imath <a,x>} = \sum_{a\in \mathbb{Z}^2}c_ae^{\imath <a,x>}
\]
which implies by Lemma 2.1 that $p_a(x)$ is a constant for all $a$ in $\mathbb{Z}^2$, and that $(a_1+\ell a_2)p_a = c_a$.

 As $\ell$ is Liouville, it is irrational, hence $a_1+\ell a_2 \neq 0$ for all $a=(a_1,a_2) \neq (0,0)$. It follows that the $p_a$ are equal to $\frac{c_a}{a_1+\ell a_2}$ for $a \neq 0$.

By assumption this solution belongs to $\mathcal{C}^\infty (T)[x_1,x_2]$  for every $g$ in $\mathcal{C}^\infty(\t)$ and thus for every choice 
of the $\{c_a\}$ that are rapidly decreasing. It would then follow that $|a_1+\ell a_2|\geq c(1+|a_1|+ |a_2|)^{-N}$ for some $N\geq 1$, 
some $c>0$ and all $(a_1 ,a_2)\in \mathbb{Z}^2$. This is a contradiction, for  since $\ell$ is a Liouville number, there cannot be such a bound.  \end{proof}
  
\section{Signal spaces for periodic 1D systems}

 In this section $\t=\mathbb{R}/2\pi \mathbb{Z}$ and $\mathcal{D}=\mathbb{C}[D]$ with
 $D =\frac{1}{\imath}\frac{d}{dx}$. We compute for various signal spaces $\mathcal{F}$ the Willems closure $\mathcal{M}^{\perp \perp}$ of a module $\mathcal{M}\subset \mathcal{D}^q$. Write 
$(\mathcal{D}^q/\mathcal{M})_{\tor}=\oplus \mathcal{D}/(D-a_i)^{n_i}$. By Lemma 1.1, 
$\mathcal{M}^{\perp \perp}/\mathcal{M}\subset (\mathcal{D}^q/\mathcal{M})_{\tor}$, and using Corollary 1.1 and  Lemma 2.1
it follows that:
\begin{enumerate}
\item For $\mathcal{F}=\mathcal{C}^\infty(\t)$, \[ \mathcal{M}^{\perp\perp}/\mathcal{M}=
(\oplus_{a_i\not \in \mathbb{Z} } \mathcal{D}/(D-a_i)^{n_i})\oplus 
( \oplus _{a_i\in \mathbb{Z} } (D-a_i)\mathcal{D}/(D-a_i)^{n_i})\]
\item For $\mathcal{F}=\mathcal{C}^\infty (\t)[x]$, or for the injective envelope of $\mathcal{C}^\infty (\t)$ in it,  
\[ \mathcal{M}^{\perp\perp}/\mathcal{M}= \oplus_{a_i\not \in \mathbb{Z} } \mathcal{D}/(D-a_i)^{n_i}\]
and
$\mathcal{M}^{\perp \perp}$ consists of the elements $r\in \mathcal{D}^q$ such that $Lr$ is in $\mathcal{M}$ for an $L\in \mathcal{D}$ without zeros in $\mathbb{Z}$.
\item For $\mathcal{F}=\mathcal{C}^\infty (\pt)$,  
\[ \mathcal{M}^{\perp\perp}/\mathcal{M} =
(\oplus_{a_i\not \in \mathbb{Q} } \mathcal{D}/(D-a_i)^{n_i})\oplus 
( \oplus _{a_i\in \mathbb{Q} } (D-a_i)\mathcal{D}/(D-a_i)^{n_i})\]
\item For $\mathcal{F}=\mathcal{C}^\infty (\pt)[x]$, or for the injective envelope of $\mathcal{C}^\infty (\pt)$
in it,
\[ \mathcal{M}^{\perp\perp}/\mathcal{M} =
\oplus_{a_i\not \in \mathbb{Q} } \mathcal{D}/(D-a_i)^{n_i}\]
and 
$\mathcal{M}^{\perp \perp}$ consists of the elements $r\in \mathcal{D}^q$ such that
$Lr$ is in $\mathcal{M}$ for an $L\in \mathcal{D}$ without zeros in $\mathbb{Q}$.
\end{enumerate}

Case (2) can be rephrased by
stating that $\mathcal{M}=\mathcal{M}^{\perp \perp }$ if and only if the support of the module 
$(\mathcal{D}^q/\mathcal{M})_{\tor}$ lies in $\mathbb{Z}\subset \mathbb{A}^1=\mathbb{C}$. 
The signal space $\mathcal{F}=\mathcal{C}^\infty(\t)[x]$ gives rise to a rather restricted set
of behaviors in $\mathcal{F}^q$. Indeed, the modules $\mathcal{M} = \mathcal{M}^{\perp \perp}$ corresponding
to behaviors in $\mathcal{F}^q$
are of the form $L\cdot \mathcal{W}\subset \mathcal{M}\subset \mathcal{W}$, where $\mathcal{W}$ is a direct summand of
$\mathcal{D}^q$ and $L\in \mathcal{D},\ L\neq 0$ has all its zeros in $\mathbb{Z}$.

Case (4) can be rephrased by stating that $\mathcal{M}=\mathcal{M}^{\perp \perp}$ if and only if the support of
$(\mathcal{D}^q/\mathcal{M})_{\tor}$ lies in $\mathbb{Q}\subset \mathbb{A}^1=\mathbb{C}$. This gives rise to a somewhat richer set of behaviors in $\mathcal{F}^q$.

\section{Signal spaces for periodic nD systems}

In this section $\t=\mathbb{R}^n/2\pi \mathbb{Z}^n$ and $n\geq 2$. For various choices of signal spaces we investigate the set of behaviors and the corresponding Willems closure. 

\subsection{$\mathcal{C}^\infty (\pt)[x_1,\dots ,x_n]$ and $\mathcal{C}^\infty(\pt)_{\fin}[x_1,\dots ,x_n]$}

According to Proposition 2.2 we may restrict ourselves in this subsection to the injective signal space $\mathcal{F}=\mathcal{C}^\infty (\pt)_{\fin}[x_1,\dots ,x_n]$.
We start with examples illustrating some of the features of the Willems closure. 

\begin{example}{\rm  $\mathcal{F}=\mathcal{C}^\infty (\pt)_{\fin}[x_1,x_2]$, $\frak{i}=(D_1^2,D_1D_2)\subset \mathcal{D} = \mathbb{C}[D_1, D_2]$.\\
For $L\in \mathcal{D}$ and $p_\lambda e^{\imath\lambda \cdot x} \in \mathcal{F}$ it follows that,
$L(p_\lambda e^{\imath\lambda \cdot x})=\{L(D_1+\lambda _1,D_2+\lambda _2)(p_\lambda )\} e^{\imath\lambda \cdot x}.$ Thus, to determine the behavior of the ideal $\frak{i}$ we have to consider the two equations 
 \[(D_1+\lambda _1)^2p_\lambda =0;\  (D_1+\lambda _1)(D_2+\lambda _2)p_\lambda =0, \mbox{  where }p_\lambda \in \mathbb{C}[x_1,x_2]\]
  If $\lambda _1\neq 0$, then the only solution is $p_\lambda =0$.\\
  If $\lambda _1=0,\ \lambda _2\in \mathbb{Q}, \lambda _2\neq 0$, then the solutions
  are $p_{(0,\lambda _2)}\in \mathbb{C}[x_2]$.\\
  If $\lambda _1=\lambda _2=0$, then the solutions are 
  $p_{(0,0)}=a_0+a_1x_1$ with $a_0\in \mathbb{C}[x_2]$ and $a_1\in \mathbb{C}$. \\
  
\noindent The behavior $\mathcal{B}=\frak{i}^\perp$ is then $\mathcal{B}_0 + 
  \mathcal{B}_1$ where 
 $\mathcal{B}_0:= \mathbb{C} + \mathbb{C}x_1$ and
 $\mathcal{B}_1:=\{ \sum _{\lambda _2\in \mathbb{Q}}  p_{(0,\lambda _2)}
 e^{\imath\lambda _2x_2}|\ \mbox{all } p_{(0,\lambda _2)}\in \mathbb{C}[x_2] \}$.\\
 
One easily sees that  $\mathcal{B}_0^\perp =(D_1^2,D_2)$, $(D_1^2,D_2)^\perp =\mathcal{B}_0$  and  $\mathcal{B}_1^\perp =(D_1)$, 
$(D_1)^\perp =\mathcal{B}_1$.  Thus $\mathcal{B}_0$ and $\mathcal{B}_1$ are
behaviors and correspond to the primary decomposition
$\frak{i}=(D_1^2,D_2)\cap (D_1)$ of the ideal $\frak{i}$. Note also that the behavior of the ideal $(D_1^2,D_2)+(D_1)=(D_1,D_2)$ is $\mathbb{C}$, which is the intersection of the behaviors $\mathcal{B}_0$ and $\mathcal{B}_1$. (The lattice structure of behaviors under the operations of sum and intersection is studied in more detail for the {\it classical spaces} in \cite{ss2}.) }
\hspace*{\fill}$\square$\end{example}
 
 \begin{example} {\rm $\mathcal{F}=\mathcal{C}^\infty (\pt)_{\fin}[x_1,x_2]$. Let $\frak{p}\subset \mathcal{D}$ denote the prime ideal generated by the operator
$L(D_1,D_2)=(D_1^2-D_2^2)+\pi (D_1D_2-1)$. 
The rational points of the variety $\mathcal{V}(\frak{p})\subset \mathbb{A}^2$ defined by the ideal 
$\frak{p}$ are $\{(1,1),(-1,-1)\}$. The behavior $\mathcal{B}:=\frak{p}^\perp \subset \mathcal{F}$ 
has the form $B_1\cdot e^{\imath(x_1+x_2)}\oplus B_{-1}\cdot e^{-\imath(x_1+x_2)}$, where
$B_1$ and $B_{-1}$ are the kernels of the operators 
$L_1:=L(D_1+1,D_2+1)=D_1^2-D_2^2+\pi D_1D_2+
(2+\pi )D_1+(-2+\pi )D_2$ and $L_2:=L(D_1-1,D_2-1)$ respectively,
acting on $\mathbb{C}[x_1,x_2]$. 

Let $\mathbb{C}[x_1,x_2]_{\leq n}$ denote the vector space of the polynomials of
total degree $\leq n$. Observe that the map 
$L_1:\mathbb{C}[x_1,x_2]_{\leq n}\rightarrow \mathbb{C}[x_1,x_2]_{\leq n-1}$
is surjective. It follows that $B_1\cap \mathbb{C}[x_1,x_2]_{\leq n}$ has dimension
$n+1$. Thus $B_1$ is an infinite dimensional subspace of $\mathbb{C}[x_1,x_2]$
and the same holds for $B_{-1}$. An explicit calculation showing $\mathcal{B}^\perp =
\frak{p}$ is possible. However, the statement $\frak{p}^{\perp \perp}=\frak{p}$ follows at once from Theorem 4.3.} \hfill $\square$ \end{example}

\begin{proposition} Let $\mathcal{F}=\mathcal{C}^\infty(\pt)_{\fin}[x_1,\dots ,x_n]$, and 
let $\mathcal{M}$ be a submodule of $\mathcal{D}^q$. Then the Willems closure $\mathcal{M}^{\perp \perp}$ of $\mathcal{M}$ with respect to $\mathcal{F}$ 
consists of the elements $x$ in $\mathcal{D}^q$ for which the ideal $\{r\in \mathcal{D}|\ rx \in \mathcal{M}\}$ is not contained
in any maximal ideal of the form $(D_1-b_1,\dots ,D_n-b_n)$ with
$(b_1,\dots ,b_n)\in \mathbb{Q}^n$. In other words $\mathcal{M}^{\perp \perp}$ is the largest submodule $\mathcal{M}^+$ of $\mathcal{D}^q$ containing $\mathcal{M}$, such that the support $S\subset \mathbb{A}^n$ of $\mathcal{M}^+/\mathcal{M}$ satisfies $S\cap \mathbb{Q}^n=\emptyset$.
\end{proposition}
\begin{proof} $\mathcal{M}^{\perp \perp }/\mathcal{M}$ consists of the elements $\xi$ such that
$\ell (\xi )=0$ for all $\ell \in \homo_\mathcal{D}(\mathcal{D}^q/\mathcal{M},\mathcal{F})$. Let 
$\frak{i}:=\{r\in \mathcal{D}|\ r\xi =0\}$. Since $\mathcal{F}$ is injective one has that 
$\xi \in \mathcal{M}^{\perp \perp}/\mathcal{M}$ if and only if $\homo_\mathcal{D}(\mathcal{D}/\frak{i},\mathcal{F})=0$. 
 
If $\frak{i}$ lies in a maximal ideal  $\frak{m}:=(D_1-b_1,\dots ,D_n-b_n)$ with $(b_1,\dots ,b_n)\in 
\mathbb{Q}^n$, then  $\homo_\mathcal{D}(\mathcal{D}/\frak{i},\mathcal{F})\neq 0$ because
${\homo}_\mathcal{D}(\mathcal{D}/\frak{m},\mathcal{F})\neq 0$.

On the other hand, suppose that  $\ell \in \homo_\mathcal{D}(\mathcal{D}/\frak{i},\mathcal{F})$ is non zero. Then
$\ell (1+\frak{i})=\sum _{a\in \mathbb{Q}^n}p_{a}(x)
e^{\imath<a,x>}$ has a non zero term 
$t:=p_b(x) e^{\imath<b,x>}$ and 
$rt=0$ for all $r\in \frak{i}$. If $b=(b_1,\dots ,b_n)$, then 
$(D_j-b_j)t=(D_j p_b(x))e^{\imath<b,x>}$. 
Thus for suitable integers $m_j\geq 0$,
$t_0:=(D_1-b_1)^{m_1}\cdots (D_n-b_n)^{m_n}t=ce^{\imath<b,x>}$ with $c\in \mathbb{C}^*$. Since $\frak{i}\cdot t_0=0$, it follows that $\frak{i}\subset (D_1-b_1,\dots ,D_n-b_n)$.
\end{proof}

A second formulation of the structure of $\mathcal{M}^{\perp \perp}$ uses the notion of {\it primary decomposition} of modules.  Let $\frak{p}\subset \mathcal{D}$ be a prime ideal. A submodule $\mathcal{M}$ of $\mathcal{D}^q$ is called $\frak{p}$-primary (with respect to $\mathcal{D}^q$) if the set $\ass(\mathcal{D}^q/\mathcal{M})$ of associated primes of $\mathcal{D}^q/\mathcal{M}$ equals $\{\frak{p}\}$. For a general submodule $\mathcal{M}\subset \mathcal{D}^q$, there exists an irredundant (one where 
no term can be  omitted) primary decomposition  $\mathcal{M}=\mathcal{M}_1\cap \dots \cap \mathcal{M}_t$ where $\mathcal{M}_i$ is $\frak{p}_i$-primary and $\{\frak{p}_1,\dots ,\frak{p}_t\}=  \ass(\mathcal{D}^q/\mathcal{M})$. For more details we refer to \cite{ma}. We note that the following theorem
is an analogue of the Nullstellensatz of \cite{ss1}. See also \cite{ps,ss2,sa} on this topic.
 
\begin{theorem} [Nullstellensatz]
Let the submodule $\mathcal{M}\subset \mathcal{D}^q$ have an irredundant primary decomposition
$\mathcal{M}=\mathcal{M}_1\cap \dots \cap \mathcal{M}_t$ where $\mathcal{M}_i$ is $\frak{p}_i$-primary. Let $\mathcal{V}(\frak{p}_i)$, the variety defined by $\frak{p}_i$,
contain a rational point for $i=1,\dots ,r$ and not for $i=r+1,\dots ,t$. Then the Willems closure $\mathcal{M}^{\perp \perp}$ with respect to 
$\mathcal{F}=\mathcal{C}^\infty(\pt)_{\fin}[x_1,\dots ,x_n]$ is equal to $\mathcal{M}_1\cap \cdots \cap \mathcal{M}_r$~. Thus $\mathcal{M}$ equals $\mathcal{M}^{\perp \perp}$ if and only if every $\mathcal{V}(\frak{p}_i)$ contains rational points.
  \end{theorem}
  
  \begin{proof}
  
  It is easy to see that $\mathcal{M}_0:=\mathcal{M}_1\cap \cdots \cap \mathcal{M}_r$ is independent of the primary decomposition (see \cite{ss1}).
We first claim that the behavior $\mathcal{M}_0^\perp$ of $\mathcal{M}_0$ in $\mathcal{F}$ equals the behavior $\mathcal{M}^\perp$ of $\mathcal{M}$. As $\mathcal{M} \subset \mathcal{M}_0$ it suffices to show that $\mathcal{M}^\perp \subset \mathcal{M}_0^\perp$. Suppose it is not. Then there is an $f$ in $\mathcal{M}^\perp$ and some $m$ in $\mathcal{M}_0 \setminus \mathcal{M}$ such that $m(D)f \neq 0$. However for every $r$ in the ideal $(\mathcal{M}:m), ~r(D)(m(D)f) = 0$. Taking Fourier transforms - every element of $\mathcal{F}$ is a temperate distribution - gives $r(x)\widehat{(m(D)f)}(x)=0$, hence the support of $\widehat{m(D)f}$ is contained in $\mathcal{V}(r)\cap\mathbb{R}^n$ for every $r$ in $(\mathcal{M}:m)$. Now $\widehat{(m(D)f)}(x)=m(x)\hat{f}(x)$, and if $f=\sum_{a\in\mathbb{Q}^n}p_a(x)e^{\imath<a,x>}$, then $\hat{f}(x)=\sum_{a\in \mathbb{Q}^n} p_a(D)\delta_a$ - where $\delta_a$ is the Dirac distribution supported at $a$ - so that the support of $\widehat{m(D)f}$ is contained in $\mathbb{Q}^n$ and hence in $\mathcal{V}((\mathcal{M}:m))\cap \mathbb{Q}^n$. 

On the other hand the ideal $(\mathcal{M}:m)$ equals $\cap_{i=1}^t(\mathcal{M}:m)$, and as $m$ is in $\mathcal{M}_0\setminus \mathcal{M}$ it follows that the radical ideal $\sqrt{(\mathcal{M}:m)}$ is equal to the intersection of a subset of $\frak{p}_{r+1},\dots,\frak{p}_t$. Thus $\mathcal{V}((\mathcal{M}:m))$ is contained in $\cup_{i=r+1}^t\mathcal{V}(\frak{p}_i)$ whose intersection with $\mathbb{Q}^n$, by assumption, is empty. Thus $m(D)f=0$, which is a contradiction to the choice of $f$ and $m$ above.   

We now show that $\mathcal{M}_0$ is the largest submodule of $\mathcal{D}^q$ with the same behavior as that of $\mathcal{M}$. So let $m$ be any element of $\mathcal{D}^q\setminus \mathcal{M}_0$, and consider the exact sequence
\[ 0 \rightarrow \mathcal{D}/(\mathcal{M}_0:m) \stackrel{m}{\longrightarrow} \mathcal{D}^q/\mathcal{M}_0 \stackrel{\pi}{\longrightarrow} \mathcal{D}^q/\mathcal{M}_0+(m) \rightarrow 0\]  
where the morphism $m$ maps the class of of $r$ to the class of $mr$, and $\pi$ is as usual. Applying the functor $\homo_\mathcal{D}(\cdot ~ ,\mathcal{F})$ gives the exact sequence
\[0\rightarrow \homo_\mathcal{D}(\mathcal{D}^q/\mathcal{M}_0+(m),\mathcal{F}) \longrightarrow \homo_\mathcal{D}(\mathcal{D}^q/\mathcal{M}_0,\mathcal{F}) \stackrel{m(D)}{\longrightarrow} \homo_\mathcal{D}(\mathcal{D}/(\mathcal{M}_0:m),\mathcal{F}) \rightarrow 0 \]
Observe now that $\mathcal{V}((\mathcal{M}_0:m))$ is the union of some of the varieties $\mathcal{V}(\frak{p}_1), \dots \mathcal{V}(\frak{p}_r)$, hence by assumption there is a rational point, say $a$ on it. Therefore the function $e^{\imath{<a,x>}}$ is in the last term $\homo_\mathcal{D}(\mathcal{D}/(\mathcal{M}_0:m),\mathcal{F})$ above and which is therefore nonzero. This implies that the behavior $(\mathcal{M}_0+m)^\perp$ is strictly smaller than the behavior $\mathcal{M}^\perp$.
 \end{proof}
 
A central notion of the subject is that of a {\it controllable} behavior \cite{w,ps}. A behavior which admits an {\it image representation} is controllable and the next result characterizes such behaviors.

\begin{theorem} Let $\mathcal{F}=\mathcal{C}^\infty(\pt)_{\fin}[x_1,\dots ,x_n]$. Then the behavior $\mathcal{M}^\perp$ in $\mathcal{F}^q$ of a submodule $\mathcal{M} \subset \mathcal{D}^q$ is the image of some morphism $L(D):\mathcal{F}^p \rightarrow \mathcal{F}^q$ if and only if the varieties of the nonzero associated primes of $\mathcal{D}^q/\mathcal{M}$ do not contain rational points.
\end{theorem}
\begin{proof}
Let $M(D)$ be an $r\times q$ matrix whose $r$ {\it rows} generate $\mathcal{M}$ (so that $\mathcal{M}^\perp$ equals the kernel of the morphism $M(D):\mathcal{F}^q \rightarrow \mathcal{F}^r$). Let $\mathcal{L}$ be the submodule of $\mathcal{D}^q$ consisting of all relations between the $q$ columns of $M(D)$. Suppose that $\mathcal{L}$ is generated by some $p$ elements $\ell_1,\dots ,\ell_p$. Let $L(D)$ be the matrix whose {\it columns} are $\ell_1,\dots ,\ell_p$ and which therefore defines a morphism $L(D):\mathcal{F}^p \rightarrow \mathcal{F}^q$. As $\mathcal{F}$ is an injective module, its image equals the kernel of a morphism $M_1(D):\mathcal{F}^q \rightarrow \mathcal{F}^{r_1}$, where that $r_1$ rows of $M_1(D)$ generate all relations between the rows of $L(D)$. Let $\mathcal{M}_1$ be the submodule of $\mathcal{D}^q$ generated by the rows of $M_1(D)$; then
$\mathcal{M}_1/\mathcal{M}=(\mathcal{D}^q/\mathcal{M})_\tor$ so that $\mathcal{D}^q/\mathcal{M}_1$ is torsion free. Thus it follows that $\mathcal{M}^\perp$ is an image, in fact the image of $L(D):\mathcal{F}^p \rightarrow \mathcal{F}^q$, if and only if $\mathcal{M}^\perp=\mathcal{M}_1^\perp$, i.e if and only if the Willems closure of $\mathcal{M}$ equals $\mathcal{M}_1$. By the previous theorem this is so if and only if the variety of every nonzero associated prime of $\mathcal{D}^q/\mathcal{M}$ does not contain rational points.\end{proof}

\subsection{$\mathcal{C}^\infty (\t)[x_1,\dots ,x_n]$ and $\mathcal{C}^\infty(\t)_{\fin}[x_1,\dots ,x_n]$}
In this case it suffices to consider the signal space $\mathcal{C}^\infty(\t)_{\fin}[x_1,\dots ,x_n]$.
The results of \S 4.1, as well as the examples, carry over if everywhere one replaces
$\mathbb{Q}$ by $\mathbb{Z}$ and `rational point' by `integral point'.
 
\subsection{$\mathcal{C}^\infty (\pt)$ and $\mathcal{C}^\infty(\pt)_{\fin}$}

We consider the signal space $\mathcal{F}=\mathcal{C}^\infty(\pt)_{\fin}$.\\ 
\noindent {\it Description of $\frak{i}^{\perp \perp}$ for ideals $\frak{i}\subset \mathcal{D}$ and behaviors in $\mathcal{F}$}: 
Recall that the support of a series $f(x)=\sum _{a\in \mathbb{Q}^n}c_a
e^{\imath<a,x>}$ is the set $\{a|\ c_a\neq 0\}$.
For $a=(a_1,\dots ,a_n)\in \mathbb{C}^n$ we write 
$(D-a)$ for the maximal ideal $(D_1-a_1,\dots ,D_n-a_n)$. Given an ideal $\frak{i}\subset \mathcal{D}$, let $\mathcal{V}(\frak{i})$ be its variety in $\mathbb{C}^n$, and let $\mathcal{S}(\frak{i})=\mathcal{V}(\frak{i})(\mathbb{Q})$ (i.e., $\mathcal{V}(\frak{i})\cap \mathbb{Q}^n$ seen as a subset of $\mathbb{C}^n$).
 
 If $f(x)=\sum _{a\in \mathbb{Q}^n}c_a
e^{\imath<a,x>}\in \frak{i}^\perp$, then each $c_a
e^{\imath<a,x>}\in \frak{i}^\perp$. Thus $f\in \frak{i}^\perp$ if and only if the support
of $f$ lies in $\mathcal{S}(\frak{i})$. Further, $\frak{i}^{\perp \perp}$ consists of all the polynomials in $\mathcal{D}$
which are zero on the set $\mathcal{S}(\frak{i})$. In other words 
$\frak{i}^{\perp \perp}=\bigcap _{a\in \mathcal{S}(\frak{i})} (D-a)$.
Equivalently, $\frak{i}^{\perp \perp}$ is the reduced ideal of the Zariski closure of $\mathcal{S}(\frak{i})$.
 The behaviors $\mathcal{B}\subset \mathcal{F}$ are in this way in 1-1 correspondence with Zariski closed subsets $S$ of $\mathbb{C}^n$ satisfying
$S\cap \mathbb{Q}^n$ is Zariski dense in $S$.  

\bigskip 

\noindent {\it Description of $\mathcal{M}^{\perp \perp}$ for submodules $\mathcal{M}$ of $\mathcal{D}^q$}:
The elements of $\mathcal{F}^q$ are written in the  
form $f(x)=\sum _{a\in \mathbb{Q}^n}c_ae^{\imath<a,x>}$, with $c_a=(c_{a_1},\dots ,c_{a_q}))\in \mathbb{C}^q$. Now 
$m=(m_1,\dots ,m_q)\in \mathcal{D}^q$ applied to $f$ has the form
$\sum _{a\in \mathbb{Q}^n} <m(a),c_a>
e^{\imath<a,x>}$, with 
$ <m(a),c_a>=\sum _{j=1}^q m_j(a)c_{a_j}$. (Here, for any 
$m=(m_1,\dots ,m_q)\in \mathcal{D}^q$ we write 
$m(a)=(m_1(a),\dots ,m_q(a))\in \mathbb{C}^q$, where as before $m_i(a) = m_i(a_1,\dots, a_n)$.

  For a fixed $a\in \mathbb{Q}^n$, the set 
  $V(a):=\{m(a)\in \mathbb{C}^q|\ m\in \mathcal{M}\}$ is a 
linear subspace of $\mathbb{C}^q$. We conclude that $\mathcal{M}^\perp$ consists of the
elements $f(x)=\sum _{a\in \mathbb{Q}^n}c_ae^{\imath<a,x>}$ such that $<V(a),c_a>=0$. It now follows that
$\mathcal{M}^{\perp \perp}$ consists of the elements $r\in \mathcal{D}^q$ such that for each
$a\in \mathbb{Q}^n$, $r(a)\in V(a)$.\\

\begin{example}{\rm
$n=2,\ q=2$ and $\mathcal{M}\subset \mathcal{D}^2$ is generated by $(D_1^2,D_1D_2)$. Then $V(a)=\mathbb{C}(a_1^2,a_1a_2)$
for all $a\in \mathbb{Q}^2$. One finds that 
$\mathcal{M}^\perp\subset \mathcal{F}^2$ consists of the expressions
$\sum _{a\in \mathbb{Q}^2}(c_{a_1},c_{a_2})e^{\imath<a,x>}$ satisfying $a_1^2c_{a_1}+a_1a_2c_{a_2}=0$. Further $\mathcal{M}^{\perp \perp }=\mathcal{M}$.} \hfill $\square$ \end{example}

\noindent {\it An `algorithm' computing  $\mathcal{M}^{\perp \perp}$ for a submodule $\mathcal{M}$ of $\mathcal{D}^q$}:
For every $b\in \mathbb{Q}^n$ one considers  the
homomorphism \[m_b: \mathcal{F}=C^\infty(\pt)_{\fin}\rightarrow
\mathbb{C}e^{\imath<b,x>}\cong \mathcal{D}/(D-b)
\]
given by $m_b:\sum _ac_ae^{\imath<a,x>} 
\mapsto c_be^{\imath<b,x>}$ (where as before $(D-b)=(D_1-b_1,\dots,D_n-b_n)$). It follows at once that
$\xi \in \mathcal{D}^q/\mathcal{M}$ belongs to $\mathcal{M}^{\perp \perp}/\mathcal{M}$ if and only if $\ell (\xi)=0$
for every homomorphism $\ell :\mathcal{D}^q/\mathcal{M}\rightarrow \mathcal{D}/(D-b)$
with $b\in \mathbb{Q}^n$. As in the proof of Theorem 4.1, we consider an
irredundant primary decomposition $\cap \mathcal{M}_i$  of $\mathcal{M}$ and try to compute the $\mathcal{M}_i^{\perp \perp}$.\\

Let $\mathcal{M}$ be $\frak{p}$-primary for its embedding in $\mathcal{D}^q$, then $\mathcal{M}^{\perp \perp}\supset \mathcal{M}+\frak{p}\mathcal{D}^q$ and we may
replace $\mathcal{M}$ by the $\frak{p}$-primary module $\mathcal{M}_1:=\mathcal{M}+\frak{p}\mathcal{D}^q$ since $\mathcal{M}^{\perp \perp }=\mathcal{M}_1^{\perp \perp}$.
 We observe that  $\mathcal{D}^q/\mathcal{M}_1$ has, as a module over $\mathcal{D}/\frak{p}$, no torsion and therefore is a submodule of $(\mathcal{D}/\frak{p})^r$ for some $r\geq 1$. Now $\mathcal{M}_1^{\perp \perp}/\mathcal{M}_1= \cap (\ker (\mathcal{D}^q/\mathcal{M}_1\stackrel{\ell}{\rightarrow}\mathcal{D}/(D-b))$, where the intersection is taken over all $b\in \mathcal{V}(\frak{p})\cap \mathbb{Q}^n$ and all homomorphisms $\ell$.

\bigskip

Suppose that  the set $\mathcal{V}(\frak{p})\cap \mathbb{Q}^n$ is Zariski dense in $V(\frak{p})$ (this holds in particular for $\frak{p}=(0)$).
Then $\cap (\ker ((\mathcal{D}/\frak{p})^r\stackrel{\ell}{\rightarrow} \mathcal{D}/(D-b))$,    
$b\in \mathcal{V}(\frak{p})\cap \mathbb{Q}^n$ and all $\ell$, equals $(0)$. It follows that $\mathcal{M}_1^{\perp \perp }=\mathcal{M}_1$.\\

Suppose that  the set $\mathcal{V}(\frak{p})\cap \mathbb{Q}^n$ is empty, then $\mathcal{M}_1^{\perp \perp}=\mathcal{D}^q$.\\

Suppose that  the set $S:=\mathcal{V}(\frak{p})\cap \mathbb{Q}^n$ is not empty and is not dense in $\mathcal{V}(\frak{p})$. The radical      
ideal $\frak{i}:=\cap _{b\in S}(D-b)$ defines  $\mathcal{V}(\frak{i})\subset \mathbb{C}^n$, which is
the closure of $S$.  Now $\cap (\ker (\mathcal{D}^q/\mathcal{M}_1\stackrel{\ell}{\rightarrow}\mathcal{D}/(D-b))$, where the intersection is taken over all $b\in \mathcal{V}(\frak{i})\cap \mathbb{Q}^n$ and all $\ell$, contains $\frak{i}\mathcal{D}^q$. 
Thus we may as well continue with the module $\mathcal{M}_2:=\mathcal{M}_1+\frak{i}\mathcal{D}^q$  since $\mathcal{M}_2^{\perp \perp}=\mathcal{M}_1^{\perp \perp}$.

In general, $\mathcal{M}_2$ is not primary and we have to replace $\mathcal{M}_2$ again by the elements of an irredundant primary
$\cap (\mathcal{M}_2)_i $ decomposition of $\mathcal{M}_2$.   The minimal prime ideals $\frak{q}$ containing $\frak{i}$ are associated primes of 
$\mathcal{D}^q/\mathcal{M}_2$. For the corresponding primary factor $(\mathcal{M}_2)_i$ one has $(\mathcal{M}_2)_i^{\perp \perp}=
(\mathcal{M}_2)_i$ because $\mathcal{V}(\frak{q})\cap \mathbb{Q}^n$ is dense in $\mathcal{V}(\frak{q})$. If there are no more primary factors
(or if the other primary factors belong to prime ideals $\frak{r}$ such that $\mathcal{V}(\frak{r})\cap \mathbb{Q}^n$ is dense in $\mathcal{V}(\frak{r})$), then $\mathcal{M}_2^{\perp \perp}=\mathcal{M}_2$, and we are finished. However, if $\mathcal{M}_2$ has a primary factor $\mathcal{M}_3$ corresponding
to a prime ideal $\frak{r}$ such that $\mathcal{V}(\frak{r})\cap \mathbb{Q}^n$ is not dense in $\mathcal{V}(\frak{r})$, then we 
have to repeat the above process. The Noether property guaranties that the process ends.  Except for the
problem of finding rational points on irreducible subspaces of $\mathbb{A}^n$, the above is really an algorithm.   \\

\begin{example} Behaviors related to rational points on algebraic  varieties.{\rm \\ 
 (1) $n=2$. $\frak{i}=(D_1^2+D_2^2-1)\subset \mathcal{D}\ $
yields $\frak{i}^\perp =\{\sum _{a\in \mathbb{Q}^2,\ a_1^2+a_2^2=1}
c_ae^{\imath<a,x>}\}$  and $\frak{i}^{\perp \perp }=\frak{i}$.\\
(2) $n=2$. $\frak{i}=(D_1^2-(D_1^3+aD_1^2+bD_1+c))\subset 
\mathcal{D}.\ $ We suppose that $a,b,c\in \mathbb{Q}$ and that the equation defines an affine elliptic curve. Now these are the following possibilities (see \cite{sil}):\\
(a) The elliptic curve has no rational point other than its infinite point. Then
$\frak{i}^{\perp\perp}=\mathcal{D}$.\\
(b) The elliptic curve has finitely many rational points. Then 
$\frak{i}^{\perp \perp }\subset \mathcal{D}$ is the intersection of the finitely many maximal ideals
$(D-a)$ with $a\in \mathbb{Q}^2$ lying on
the elliptic curve.\\ 
(c) The rank of the elliptic curve is positive and $\frak{i}^{\perp \perp }=\frak{i}$.\\
(3) $n=3$. Let the principal prime ideal $\frak{p}\subset \mathcal{D}$ define an irreducible affine surface $S\subset \mathbb{A}^3$ over $\mathbb{Q}$. The following possibilities
occur:\\    
(a) $S(\mathbb{Q})=\emptyset$ and $\frak{p}^{\perp \perp }=\mathcal{D}$.\\
(b) $S(\mathbb{Q})$ is finite (and non empty); then $\frak{p}^{\perp \perp}$ is the intersection
of the maximal ideals $(D-a) $ with $a\in
S(\mathbb{Q})$.\\
(c) $S(\mathbb{Q})$ is infinite and the Zariski closure of this set is a curve on
$S$. Then $\frak{p}^{\perp \perp}$ is the (radical) ideal of this curve.\\
(d) $S(\mathbb{Q})$ is Zariski dense in $S$; then
 $\frak{p}^{\perp \perp }=\frak{p}$.}\hfill $\square$\end{example}

\subsection{$\mathcal{C}^\infty (\t)$ and $\mathcal{C}^\infty(\t)_{\fin}$}
We consider the signal space $\mathcal{F}=\mathcal{C}^\infty(\t)_{\fin}$. As in \S 4.3, there is
a 1-1 relation between the behaviors $\mathcal{B}\subset \mathcal{F}$ and the Zariski
closed subsets $S$ of $\mathbb{C}^n$ such that $S\cap \mathbb{Z}^n$ is dense in
$S$. For an ideal $\frak{i}\subset \mathcal{D}$, the ideal $\frak{i}^{\perp \perp}$ is the intersection
of the maximal ideals $(D-a)\supset \frak{i}$ with
$a\in \mathbb{Z}^n$.  The descriptions of $\mathcal{M}^{\perp \perp}$ for a submodule $\mathcal{M}$ of $\mathcal{D}^q$ are the ones given in \S 4.3 with $\mathbb{Z}$ replacing 
$\mathbb{Q}$.\\

\begin{example} {\rm
(1) Let $\frak{i}\subset \mathcal{D}$ be an ideal. Let $\frak{j}\subset \mathcal{D}$ denote  the smallest ideal containing $\frak{i}$ which is generated by elements in $\mathbb{Z}[D_1,\dots ,D_n]$. Then $\frak{i}^\perp =\frak{j}^\perp$. Indeed, $\frak{i}^{\perp \perp}$ is generated by
elements in $\mathbb{Z}[D_1,\dots ,D_n]$.
Consider for example the ideal $\frak{i}\subset \mathbb{C}[D_1,D_2,D_3]$ generated by $(D_1^2-D_2^2)+\pi (D_1^2+D_3^3)+
\pi ^2(D_1D_2D_3-1)$. The ideal $\frak{j}$ is generated by
$(D_1^2-D_2^2), (D_1^2+D_3^3),
(D_1D_2D_3-1)$. Then $\frak{i}^\perp =\frak{j}^\perp$ and 
$S(\frak{i})=\{(1,-1,-1), (-1,1,-1)\}$. \\
(2) Let $\frak{i}\subset \mathbb{C}[D_1,D_2,D_3]$ be generated by
$D_1^2+D_2^2-D_3^2$. Then $\frak{i}^{\perp \perp}=\frak{i}$ because the
set $S(\frak{i})=\{(a_1,a_2,a_3)\in \mathbb{Z}^3|\ a_1^2+a_2^2-a_3^2=0 \}$ is Zariski dense in $\{(a_1,a_2,a_3)\in \mathbb{C}^3|\ a_1^2+a_2^2-a_3^2=0 \}$.
}
\hfill $\square$
\end{example}

\noindent {\bf Acknowledgement}: We are grateful to the referees for their careful reading of the manuscript.

\end{document}